\UseRawInputEncoding
\documentclass{article}
\usepackage[affil-it]{authblk}
\usepackage[utf8]{inputenc}
\usepackage[english]{babel}
\usepackage{amsmath,amssymb,amsthm,amsfonts,bm,cite}
\newtheorem{thm}{Theorem}[section]
\newtheorem{theorem}[thm]{Theorem}

\newtheorem{definition}{Definition}[section]

\newtheorem{rem}{Remark}[section]
\newtheorem{corollary}{Corollary}[section]

\usepackage{graphicx}

\title{Measuring Abundance with Abundancy Index}

\begin{document}
%
\pagestyle{headings}  

\title{Measuring Abundance with Abundancy Index}
%
\author{Kalpok Guha%
  \thanks{Email address: \texttt{kalpok.guha@gmail.com}; Corresponding author}}
\affil{Presidency University, Kolkata}

\author{Sourangshu Ghosh%
  \thanks{Email address: \texttt{sourangshug123@gmail.com}}}
\affil{Indian Institute of Technology Kharagpur, India}

\date{}
\maketitle 

\noindent 
\begin{abstract} A positive integer \(n\) is called perfect if \(\sigma(n)=2n\), where $\sigma (n)$ denote the sum of divisors of $n$. In this paper we study the ratio \(\frac{\sigma(n)}{n}\). We define the function Abundancy Index \(I:\mathbb{N} \to \mathbb{Q}\) with \(I(n)=\frac{\sigma(n)}{n}\). Then we study different properties of Abundancy Index and discuss the set of Abundancy Index. Using this function we define a new class of numbers known as superabundant numbers.  Finally we study superabundant numbers and their connection with Riemann Hypothesis.
\end{abstract}

\section{Introduction}
\begin{definition}
A positive integer $n$ is called perfect if $\sigma (n)=2n$, where $\sigma (n)$ denote the sum of divisors of $n$.
\end{definition}

The first few perfect numbers are $6,28,496,8128,...$ (OEIS A000396), This is a well studied topic in number theory.  Euclid studied properties and nature of perfect numbers in 300 BC. He proved that if $2^p-1$ is a prime, then $2^{p-1}(2^p-1)$ is an even perfect number(Elements, Prop. IX.36). Later mathematicians have spent years to study the properties of perfect numbers. But still many questions about perfect numbers remain unsolved. Two famous conjectures related to perfect numbers are 
\begin{enumerate}
    \item There exist infinitely many perfect numbers. Euler\cite{Euler} proved that a number is an even perfect numbers iff it can be written as $2^{p-1}(2^p-1)$ and $2^p-1$ is also a prime number. Primes numbers of the form $2^p-1$ are known as Mersenne primes. Therefore this conjecture is equivalent to the conjecture that there exist infinitely many Mersenne primes. Some good references on this topic are \cite{Gerstein}, \cite{Caldwell},  \cite{Travaglini}.
    \item There do not exist exist any odd perfect numbers. Computation of Lower Bounds for the smallest perfect numbers have been done by many mathematicians. Kanold (1957)\cite{Kanold} found the bound $10^{20}$, Tuckerman (1973) \cite{Tuckerman} found the bound $10^{36}$, Hagis (1973) \cite{Hagis}found the bound $10^{50}$, Brent and Cohen (1989) \cite{Brent1} found the bound $10^{160}$, Brent et al. (1991) \cite{Brent2} found the bound $10^{300}$. The best bound till today is $10^{1500}$ by Ochem and Rao (2012)\cite{Ochem}. The odd perfect numbers if exist must be of the form $p^{4\lambda +1}Q^2$, where $p$ is a prime of the form $4n+1$ as proven by Euler\cite{Burton}\cite{Weisstein}.Touchard\cite{Touchard} and Holdener\cite{Holdener} proved that the odd perfect numbers if exist must be of the form $12k+1$ or $36k+1$. Stuyvaert\cite{Dickson} proved that the odd perfect numbers if exist must be must be a sum of two squares. Greathouse and Weisstein\cite{Greathouse} alternatively writes that any odd perfect number must be of the form
    $$N=p^{\alpha}{q_1}^{2\beta_1}...{q_r}^{2\beta_r}$$
    where all the primes are odd. Also $p\equiv\alpha\equiv1(\bmod 4)$. Steuerwald\cite{Steuerwald} and Yamada\cite{Yamada}  proved that all the $\beta_i$s cannot be 1. Odd perfect numbers have a large number of distinct prime factors. The odd perfect number if exist must have at least 6 distinct prime factors, as proved by Gradshtein\cite{Ball}. This was extended to 8 by Haggis\cite{Haggis1}. If there are 8 the number must be divisible by 15, as proved by Voight \cite{Voight}. Norton\cite{Norton} proved that odd perfect numbers must have at least 15 and 27 distinct prime factors if the number is not divisible by 3 or 5 and 3, 5, or 7 respectively. Neilsen\cite{Neilsen} extended the bound by showing that odd perfect numbers should have at least 9 distinct prime factors and if it is not divisible by 3 it should have at least 12 distinct prime factors. Hare\cite{Hare} shown that any odd perfect number must have at least 75 prime factors. The method used by Hare involves factorization of several large numbers\cite{Weisstein}\cite{Hare}.The best lower bound is by Ochem and Rao (2012)\cite{Ochem}, who prove that any odd perfect number must have at least 101 prime factors. Odd perfect numbers must have the largest prime factor very large. The first such lower bound was proved by Haggis\cite{Haggis2}, who proved every Odd Perfect Number has a Prime Factor which exceeds $10^{6}$. Iannucci\cite{Iannucci1}\cite{Iannucci2}, Jenkins\cite{Jenkins}, Goto and Ohno\cite{Goto} proved that the largest three factors must be at least 100000007, 10007, and 101\cite{Weisstein}.
\end{enumerate}
Two other related concepts are abundant numbers and deficient numbers.
\begin{definition}
A positive integer $n$ is called an abundant number if $\sigma(n)>2n$.
\end{definition}
\begin{definition} 
A positive integer $n$ is called a deficient number if $\sigma(n)<2n$.
\end{definition}
To study these interesting properties of these beautiful numbers we define \textbf{Abundancy Index}. That was defined by Laatsch\cite{Laat}.
\begin{definition} For a positive integer $n$, the Abundancy index $I(n)$ is defined as $I(n)=\frac{\sigma(n)}{n}$.
\end{definition}
 More generally Abundancy Index can be considered as a measure of perfection of an integer. We can easily observe a positive integer is perfect when $I(n)=2$ and $n$ is abundant or deficient when $I(n)>2$ or $I(n)<2$ respectively. Positive integers with integer valued Abundancy indices are called \textbf{multiperfect numbers}. In this article we study different properties about  Abundancy Index and to try generalize the Abundancy index of any positive integer $n$.

\section{Properties}

\begin{theorem} The abundancy index function $I(n)$ is a multiplicative function.
\end{theorem}
\begin{proof} 
Let \(m,n\) be any two co-prime positive integers. Using the multiplicativity of \(\sigma\) function \cite{Burton}.
$$I(mn)=\frac{\sigma(mn)}{mn}=\frac{\sigma(m)\sigma(n)}{mn}=\frac{\sigma(m)}{m}\frac{\sigma(n)}{n}=I(m)I(n)$$
\end{proof} 

\begin{theorem}{\normalfont (Laatsch\cite{Laat})}: $I(kn) \geq I(n)$ for all $k \in \mathbb{N}$. The equality condition holds iff $k=1$.
\end{theorem}

\begin{corollary} Every proper multiple of a perfect number is abundant and every proper divisor of a perfect number is deficient.
\end{corollary}

\begin{corollary} 
There are infinitely many abundant numbers.
\end{corollary}
\begin{rem}  It is easy to see that there are infinitely many deficient numbers. Indeed, all prime numbers are
deficient, as $\sigma(p)=p+1 < 2p$.
\end{rem}

\begin{equation}
 I(n)=\sum_{d|n}\frac{1}{d},  I(n)=\frac{\sigma(n)}{n}=\frac{1}{n}\sum_{d|n}d=\frac{1}{n}\sum_{d|n}\frac{n}{d}=\sum_{d|n}\frac{1}{d}   
\end{equation}

\begin{theorem}{\normalfont (Laatsch\cite{Laat})}: The $I(n)$ is function is unbounded. 
\end{theorem}

\begin{proof} 
We discuss two proofs of this theorem. The first proof goes like this \\\\
Let $m$ be any real number. We know the series $\sum_{i=1}^{\infty}\frac{1}{i}$ is divergent. Hence for given $m$ $\exists N \in \mathbb{N} \backepsilon \sum_{i=1}^{N} \frac{1}{i} > m$. Let us take $n_{0}=$~lcm$(1,2,\cdots,N)$. Hence we get $I(n_{0})=\sum_{d|n_0}\frac{1}{d} \geq \sum_{i=1}^{N} \frac{1}{i}$. Thus for any real $m  \exists n_{0} \in \mathbb{N} \backepsilon I(n_{0})>m$. Therefore $I(n)$ is not bounded above.\\\\ 
The second proof goes like this \\

For $n_0=2\cdot3\cdots p_k=\prod_{i=1}^{k}p_i$
i.e the product of first $k$ primes.
Therefore $$I(n_0)=\prod_{i=1}^{k}(1+\frac{1}{p_i})>\sum_{i=1}^{k}\frac{1}{p_i}$$. Or, $I(n_0)>\sum_{i=1}^{k}\frac{1}{p_i}$. Now the series $\sum_{prime} \frac{1}{p}$ is divergent, as proven by Euler\cite{Euler1}. Hence we can say $I(n)$ is not bounded above.
\end{proof} 

\begin{theorem} For any \(r \in \mathbb{R}\) there are infinitely many \(n\) such that \(I(n)>r\).
\end{theorem}

\begin{proof}  By Theorem 3 we see for any $r \in \mathbb{R}\exists n_0 \in \mathbb{N}$ such that $I(n_0)>r$. By using Theorem 2 we get $I(kn_0) \geq I(n_0)$ for any positive integer $k$. Therefore $I(kn_0) > r \forall k \in  \mathbb{N}$. As there are infinitely many choices for $k$, there are infinitely many $n$ such that $I(n)>r$.
\end{proof} 

\begin{theorem} 
If $n=\prod_{i=1}^{k}p_{i}^{\alpha_i}$ where the $p_i$ are distinct primes, then $\prod_{i=1}^{k}\frac{p_i+1}{p_i} \leq I(n) \leq \prod_{i=1}^{k}\frac{p_i}{p_i-1}$
\end{theorem}

\begin{proof}
Consider $p$ to be a prime and $\alpha$ any positive integer.
 Now as proven earlier in (1), we have
 
 $$I(p^{\alpha})=\sum_{d|p^{\alpha}}\frac{1}{d}=1+\frac{1}{p}+\frac{1}{p^2}+\cdots+\frac{1}{p^{\alpha}}$$ By using the inequality
$$1+\frac{1}{p} \leq 1+\frac{1}{p}+\frac{1}{p^2}+\cdots+\frac{1}{p^{\alpha}} \leq \sum_{i=1}^{\infty}\frac{1}{p^i}$$

\noindent We get 
\begin{equation}
\frac{p+1}{p} \leq I(p^\alpha) \leq \frac{p}{p-1}    
\end{equation}

\noindent Now since $I$ is multiplicative function(Theorem 1)
\begin{equation} 
 I(n)=I(\prod_{i=1}^{k}p_{i}^{\alpha_i})=\prod_{i=1}^{k}I(p_{i}^{\alpha_{i}})
\end{equation} 
Using the inequality (1) we get $$\prod_{i=1}^{k}\frac{p_i+1}{p_i} \leq \prod_{i=1}^{k}I(p_{i}^{\alpha_i}) \leq \prod_{i=1}^{k}\frac{p_i}{p_i-1}$$
Using the identity mentioned in (3) 
$$\prod_{i=1}^{k}\frac{p_i+1}{p_i} \leq I(n) \leq \prod_{i=1}^{k}\frac{p_i}{p_i-1}$$
So we get our desired result.
\end{proof}

\section{Set of Abundancy Indices }
As we study the function \(I:\mathbb{N} \rightarrow \mathbb{Q}\), many questions arise. For example, is every rational $q \geq 1$ the Abundancy index of some integer? Many Mathematicians have tried to study the set of Abundancy indices, Laatsch \cite{Laat} shown the set \(D=\{I(n):n \geq 2\}\) is dense in \((1,\infty)\). Later Weiner\cite{Wein} showed there exists rationals which are not the Abundancy index of any integer. In 2007 Stanon and Holdener\cite{Hold} defined Abundancy Outlaw. An Abundancy outlaw is a rational greater than 1 that not an Abundancy index of integer, in other words it is not in the image map of the map $I$.
\begin{theorem}{\normalfont (Laatsch\cite{Laat})}:  \(D=\{I(n):n \geq 2\}\) is dense in \((1,\infty)\).
\end{theorem} 

\begin{definition} A rational number \(q>1\) is  said to be an Abundancy outlaw if \(I(n)=q\) has no solution in \(\mathbb{N}\).
\end{definition}

\begin{theorem}{\normalfont (Wein\cite{Wein})}: If \(k\) is relatively prime to \(m\)  and \(m < k < \sigma(m)\), then \(\frac{k}{m}\) is an Abundancy outlaw. Hence if $r/s$ is an Abundancy index with $\gcd(r, s) = 1$, then $r\geq \sigma(s)$. 
\end{theorem}
\noindent Example of such outlaws given by Holdener and Stanton \cite{Hold} are
$$5/4,7/6,9/8,10/9,11/6,11/8, 11/9, 11/10,13/8,13/10,13/12,15/14,16/15,...$$
The previous theorem was also proven by Anderson\cite{Anderson}. The theorem implies that $\frac{k + 1}{k}$ is an Abundancy index if and only if $k$ is prime,also $\frac{k + 2}{k}$ is an Abundancy outlaw whenever $k$ is an odd composite number. This is a very important result shown by Weiner, which concludes that there are rationals in \((1,\infty)\) which are not Abundancy index of any integer. This can be proven using \textbf{Theorem 3.2}.
\begin{theorem} {\normalfont (Wein\cite{Wein})}:  The set of Abundacy outlaws is dense in \((1,\infty)\). 
\end{theorem} 

In the next three theorems we are giving  few general forms of abundancy outlaw, which were studied by Holdener and Stanton \cite{Hold}. These are some particular cases of proven results by Holdener \cite{Hold1}. For the original general results someone may look into the original paper of Holdener \cite{Hold1}. \textbf{Theorem 3.4} is really just the special case of \textbf{Theorem 3.5} with $p = 2$.

\begin{theorem} For all primes \(p>3\), $$\frac{\sigma(2p)+1}{2p}$$ is an Abundancy outlaw. If $p=2$ or $p= 3$ then $\frac{\sigma(2p)+1}{2p}$ is an Abundancy index.
\end{theorem}

For $p=2$ or $p= 3$, it is easy to see that $\frac{\sigma(2p)+1}{2p}$ is an Abundancy index since $I(6)=\frac{\sigma(4)+1}{4}$ and $I(18)=\frac{\sigma(6)+1}{6}$ . By substituting $\sigma(p) =3+3p$ we can get an explicit expression. Note that $\frac{\sigma(2p)+1}{2p}=\frac{3p+4}{2p}$ is in lowest terms. Therefore if $I(n)=\frac{\sigma(2p)+1}{2p}$, then $2p|n$. Now since $p>3$, we have $I(4p)>(\sigma(2p)+1)/2p$, so $4\not|N$. Hence we have, $\sigma(2)|\sigma(N)$. Also note that since $\sigma(2p) + 1$ is not divisible by $\sigma(2) = 3$, 3 divides $N$. Therefore we can write
$$I(n)>I(6p)>2>I(4p)>\sigma(2p)+1)/2p$$
We hence arrive at a contradiction. Hence $(\sigma(2p) + 1)/2p$ is an Abundancy outlaw. Example of such outlaws given by Holdener and Stanton \cite{Hold} are 
$$\frac{19}{10},\frac{25}{14},\frac{37}{22},\frac{43}{26},\frac{55}{34},\frac{61}{38},\frac{73}{46},\frac{91}{58},\frac{97}{62},\frac{115}{74},\frac{127}{82},\frac{133}{86},\frac{145}{94},\frac{163}{106},\frac{181}{118},\frac{187}{122}...$$

\begin{theorem} For primes \(p,q\) with \(q>3\),~\(q>p\) and $\gcd(p,q+2)=\gcd(q, p+2) = 1$, $$\frac{\sigma(pq)+1}{pq}$$
is an Abundancy outlaw.
\end{theorem}

Note that if $p$ and $q=p+2$ are twin primes then \textbf{Theorem 3.5} does not hold true. We get
$$\frac{\sigma(p(p+2))+1}{p(p+2)}=\frac{\sigma(p)+1}{p}=\frac{p+2}{p}$$
Abundancy index satisfying $I(x)=\frac{p+2}{p}$ has been studied by Ryan\cite{Ryan}. It is still not known whether any such example exist. The existence of such a solution is important since if $\frac{5}{3}=\frac{3+2}{3}$ is an Abundancy index then there must exist an odd perfect number.
A similar example can be made about \textbf{Theorem 3.5} as we have done earlier for \textbf{Theorem 3.4}. For this we assume that the two odd primes $p,q$, satisfying $q\equiv 1 (\bmod p)$. Then $p\not |q+2$ and $q\not |p+$2 .Now by Dirichlet's theorem on arithmetic progressions of primes, we know that there are infinitely many such pairs of odd primes $p,q$. Example of such outlaws given by Holdener and Stanton \cite{Hold} are\\
For $p=5$
$$\frac{73}{55},\frac{193}{155},\frac{253}{205},\frac{373}{305},\frac{433}{355},\frac{613}{505},\frac{793}{655},\frac{913}{755},\frac{1093}{905},\frac{1153}{955},\frac{1273}{1055},\frac{1513}{1255},\frac{1633}{1355}$$
$$\frac{1693}{1405},\frac{1873}{1555},\frac{1993}{1655},\frac{2413}{2005},\frac{2533}{2105}, ..$$ 
For $p=7$
$$\frac{241}{203},\frac{353}{301},\frac{577}{497},\frac{913}{791},\frac{1025}{889},\frac{1585}{1379},\frac{1697}{1477},\frac{1921}{1673},\frac{2257}{1967},\frac{2705}{2359},\frac{3041}{2653},\frac{3377}{2947},\frac{3601}{3143}$$
$$\frac{3713}{3241},\frac{3937}{3437},\frac{4385}{3829},\frac{4945}{4319}..$$ 
For $p=11$
$$\frac{289}{253},\frac{817}{737},\frac{1081}{979},\frac{2401}{2189},\frac{3985}{3641},\frac{4249}{3883},\frac{4777}{4367},\frac{5041}{4609},\frac{5569}{5093},\frac{7417}{6787},\frac{7945}{7271},\frac{8209}{7513},\frac{8737}{7997}$$
$$\frac{10321}{9449},\frac{10585}{9691},\frac{11377}{10417},...$$ 

\begin{theorem}  If \(N\) is an even perfect number, then \(\frac{\sigma(2N)+1}{2N}\) is an abundancy outlaw.
\end{theorem}

\section{Superabundant Numbers}
\begin{definition} A positive integer \(n\) is called superabundant if \(I(m)<I(n)\) \(\forall m < n\).
\end{definition}

The first few superabundant numbers are 1, 2, 4, 6, 12, 24, 36, 48, 60, 120, 180. Ramanujan \cite{Rama1}\cite{Rama2}\cite{Rama3} in 1915 first introduced the idea of superabundant numbers. In 30 pages of Ramanujan's paper "Highly Composite Numbers" Ramanujan defined generalized highly composite numbers, which is a generalized case of superabundant numbers. Ramanujan's work remained unpublished till 1997 when it was published in Ramanujan Journal. The idea of Superabundant numbers were also independently defined by Alaoglu and Erdős \cite{Alag} in 1944, who are unknown to the unpublished work done by Ramanujan earlier in 1915.

\begin{theorem} There are infinitely many superabundant numbers.
\end{theorem}

\begin{proof} 
Let us assume there are finitely many superabundant numbers and \(n\) is the largest superabundant number. So $(I(m)<I(n) \mbox{~for~all~} m>n$. 
Now let us consider the integer \(2n\). By Theorem 2 we know \(I(2n)>I(n)\). So \(I(m)<I(2n)\). But \(2n\) cannot be a superabundant number. So \(\exists n_0 \backepsilon I(n_0)>I(2n)\) and \(n<n_0<2n\). Let us consider the least \(n_0\). We know $$I(n_0)>I(2n)>I(n)>I(m) \mbox{~for~all~} m < n$$
$n_0$ cannot be a superabundant number. $\exists n_1 \backepsilon I(n_0)>I(n_1)$ and $n<n_1<n_0$. It is easy to see $I(n_1)>I(2n)$ and \(n<n_1<2n\). But we had assumed $n_0$ to be least such integer. Hence we get a contradiction.
\end{proof}

So we can conclude there are infinitely many superabundant numbers.

Now we draw a connection between superabundant numbers and well known Riemann Hypothesis\cite{Riemann}, which is considered as one of the most important unsolved problems in Mathematics. Riemann Hypothesis conjectures that the Riemann zeta function defined as

$$\zeta(s)=\sum_{n=1}^{\infty}\frac{1}{n^s}=\frac{1}{1^s}+\frac{1}{2^s}+\frac{1}{3^s}+\frac{1}{4^s}+...$$

has non-trivial zeros only at the complex numbers with real part $\frac{1}{2}$. This conjecture is of significant interest to number theorists since this result has direct consequences in the distribution of prime numbers.

In 1984 Robin \cite{Robi} proved a surprising result. He showed an equivalence between Riemann Hypothesis and an bound to the Abundancy Index.

\begin{theorem} {\normalfont (Robin\cite{Robi})}: For \(n \geq 3\) we have \(I(n)<e^{\gamma}\log \log n+\frac{0.6483}{\log \log n}\).
\end{theorem}

\begin{theorem} {\normalfont (Robin\cite{Robi})}:  The Riemann Hypothesis is true if and only if \(I(n)<e^{\gamma}\log \log n\) \(\mbox{~for~all~} n \geq 5041\).
\end{theorem}

\noindent \textbf{Note:} Here \(\gamma\) denotes Euler's Gamma Constant(also known as Euler–Mascheroni constant). It is the limiting difference between the the natural logarithm and harmonic series.
$$\gamma=\lim_{x\to\infty}(-\ln x+\sum_{k=1}^{x}\frac{1}{k})$$
The value of Euler's Gamma Constant is approximately 0.57721\cite{NSloane}. \textbf{Theorem 4.3}(Robin's Inequality) is the most striking result here, it gives an alternative approach to prove or disprove Riemann's Hypothesis, one of the greatest problems in Number Theory.

This result by Robin's inequality is supported by many other findings. Gronwall \cite{Gronwall} found that 
$$\lim_{n\to\infty}\frac{\sigma(n)}{e^{\gamma}\log \log n}=1$$
Wojtowicz\cite{Wojtowicz} further showed that the values of $f=\frac{\sigma(n)}{e^{\gamma}\log \log n}$ are close to 0 on a set of asymptotic
density 1. An alternate version of Robin's inequality equivalent to Riemann Hypothesis was found by Lagarias\cite{Lagarias}, who showed the equivalence of the Riemann hypothesis to an sequence of elementary inequalities involving the harmonic numbers $H_n$, the sum of the reciprocals of the integers from 1 to $n$.
$$\sigma(n)\leq e^{H_n}\log{H_n}+H_n \mbox{~for~all} n \geq 1$$
Another alternate version of Robin's inequality is by Choie et.al \cite{Choie} who have shown that the RH holds true if and only if every natural number divisible by a fifth power greater than 1 satisfies Robin’s inequality. Briggs\cite{Brig} describe a computational study of the successive maxima of the relative sum-of-divisors function $\sigma(n)/n$. They found that the maxima of this function occur at superabundant and colossally abundant numbers and studied the density of these numbers. He then compared this with the known maximal order of $\frac{\sigma(n)}{e^{\gamma}\log \log n}$ and found out a condition equivalent to the Riemann Hypothesis using these data.

\begin{theorem} {\normalfont (Akbary\cite{Akba})}: If there is any counterexample to Robin's inequality then the least such counterexample is a superabundant number.
\end{theorem}

\begin{definition} Let \(S(x)\) be the number of superabundant numbers not exceeding \(x\).
\end{definition}

From \textbf{Theorem 4.1}, we get the inequality
$S(x)\geq \log x$, since the spacing grows at most exponentially. This gives $\log x$ as the lower bound to the counting function \(S(x)\). Note that \textbf{Theorem 4.4} helps us find a counterexample of the Robin's inequality by limiting our attention to only superabundant numbers. Unfortunately there is no algorithm find  superabundant numbers except finding it using \textbf{Definition 4.1}. Some results in the distribution of the superabundant numbers is therefore very helpful. We now state two results in that regard.

\begin{theorem} {\normalfont (Alaoglu\cite{Alag})}: \(S(x)>c\frac{\log x \log \log x}{(\log \log \log x)^2}\)  
\end{theorem}

\noindent Erdős and Nicholas \cite{Nich} proved a more stronger inequality.

\begin{theorem} {\normalfont (Nicholas\cite{Nich})}: $S(x)>(\log x)^{1+\delta}~(x>x_0)~for~every~ \delta<5/48$. 
\end{theorem}

So we finally see that abundancy index and superabundant numbers have a very close connection with Riemann Hypothesis. One may try to prove or disprove Riemann Hypothesis with the help of \textbf{Theorem 4.3}. To disprove Riemann's Hypothesis it enough to get a counterexample to Robin's inequality, one might try to find it computationally and \textbf{Theorem 4.4} will definitely make his or her job easier.

\end{document}